\documentclass[10pt]{amsart}

\usepackage[letterpaper, margin=1in]{geometry}

\usepackage{amsmath, amssymb, amsthm}  

\usepackage[plainpages = false, pdfpagelabels, 
						 bookmarks,
						 bookmarksopen = true,
						 bookmarksnumbered = true,
						 breaklinks = true,
						 linktocpage,
						 colorlinks = true,
						 linkcolor = blue,
						 urlcolor  = blue,
						 citecolor = red,
						 anchorcolor = green,
						 hyperindex = true,
						 hyperfigures
						 ]{hyperref}
\usepackage[utf8]{inputenc}
\usepackage[english]{babel}
\usepackage[T1]{fontenc}
\usepackage[shortlabels]{enumitem}
\usepackage{environ}

\usepackage{mathtools}

\NewEnviron{spliteq}{
\begin{equation}\begin{split}
  \BODY
\end{split}\end{equation}
}
\NewEnviron{spliteq*}{
\begin{equation*}\begin{split}
  \BODY
\end{split}\end{equation*}
}

\newcommand{\CC}{\ensuremath{\mathbb{C}}}      
\newcommand{\RR}{\ensuremath{\mathbb{R}}}

\newcommand{\QQ}{\ensuremath{\mathbb{Q}}}   
\newcommand{\ZZ}{\ensuremath{\mathbb{Z}}}

\newcommand{\eps}{\varepsilon}

\DeclareMathOperator{\re}{Re}

\newtheorem{theorem}{Theorem}
\newtheorem*{theorem*}{Theorem}

\newtheorem*{proposition*}{Proposition}

\newtheorem*{lemma*}{Lemma}

\newtheorem*{corollary*}{Corollary}

\theoremstyle{definition}

\newtheorem*{remark*}{Remark}
\newtheorem*{remarks*}{Remarks}

\newtheorem*{acknowledgements}{Acknowledgements}

\title{Zeros of $L(s)+L(2s)+\cdots+L(Ns)$ in the region of absolute convergence}
\author{{\L}ukasz Pa\'nkowski and Mattia Righetti}
\date{}
\address{Faculty of Mathematics and Computer Science, Adam Mickiewicz University, 61-614 Pozna\'n, Poland}
\email{lpan@amu.edu.pl}
\address{Dipartimento di Matematica, Università di Genova, Via Dodecaneso 35, 16146 Genova, Italy.}
\email{righetti@dima.unige.it}           

\begin{document}

\maketitle

\begin{abstract}
In this paper we show that for every Dirichlet $L$-function $L(s,\chi)$ and every $N\geq 2$ the Dirichlet series $L(s,\chi)+L(2s,\chi)+\cdots+L(Ns,\chi)$ have infinitely many zeros for $\sigma>1$. Moreover we show that for many general $L$-functions with an Euler product the same holds if $N$ is sufficiently large, or if $N=2$. On the other hand we show with an example the the method doesn't work in general for $N=3$.\\
%
\end{abstract}


\section{Introduction}

It is well-known that Dirichlet series without an Euler product might not satisfy an analog of the Riemann Hypothesis, and actually they often have zeros also in the region of absolute convergence where the Euler product has surely none. One simple example is the so-called Davenport-Heilbronn zeta function (see e.g. \cite[\S10.25]{titchmarsh}), which is defined by
\[
	f(s)=\frac{1-i\varkappa}{2}L(s,\chi_1)+\frac{1+i\varkappa}{2}L(s,\overline{\chi}_1),\qquad\varkappa=\frac{\sqrt{10-2\sqrt{5}}-2}{\sqrt{5}-1},
\]
where $\chi_1$ denotes a Dirichlet character modulo $5$ such that $\chi_1(2)=i$ and $L(s,\chi)$ denotes the Dirichlet $L$-function attached to $\chi$. One can easily observe that it is a Dirichlet series satisfying the Riemann-type functional equation, but it has no Euler product and infinitely many zeros for $\sigma>1$, so any analog of the Riemann Hypothesis cannot hold. Moreover in \cite{davenport1,davenport2} Davenport and Heilbronn provided us with more examples of zeta-functions with infinitely many zeros in the region of absolute convergence. For instance, they proved that the Hurwitz zeta-functions $\zeta(s,\alpha)$, $\alpha\not\in\frac{1}{2}\ZZ $, has infinitely many zeros in the half-plane $\sigma>1$, unless $\alpha$ is algebraic irrational, whereas the same property in this case was proved by Cassels in \cite{cassels}. It is noteworthy to mention that the fact that $\zeta(s,\alpha)$ with rational $\alpha$ can be written as a combination of Dirichlet $L$-functions plays a crucial role in the reasoning of Davenport and Heilbronn. Recently there have been some generalizations for linear combinations of $L$-functions starting with Saias and Weingartner \cite{saias}, and then Booker and Thorne \cite{booker} and Righetti \cite{righetti}. Note that in all aforementioned papers it was crucial to have a combination of $L$-functions with an abscissa of absolute convergence equal to $1$. Combinations of two $L$-functions, where at least one term has an abscissa of absolute convergence 1 were partially investigated by Nakamura and Pa\'nkowski in \cite{nakapank}. Let us notice that the case of a combination of only two terms in much easier as it is sufficient to consider the value-distribution of a quotient of two $L$-functions, which has good property like an Euler product. Thus, in order to treat combinations of $L$-functions with three and more terms some new idea is needed. 

In this paper we consider the following related problem: given an $L$-function $L(s)$ with an Euler product, do the Dirichlet series $L(s)+L(2s)+\cdots+L(Ns)$ have infinitely many zeros in the region of absolute convergence?

The following result shows that the answer is affirmative for most of the known $L$-functions for every sufficiently large $N$, or if $N=2$, and in some cases we are able to show the existence of zeros for every $N\geq 2$.

\begin{theorem}\label{theorem:L1n}
Let $L(s)$ be a Dirichlet series with the following properties:
\begin{enumerate}[(I)]
    \item\label{hp:abs_conv} $L(s)=\displaystyle\sum_{n=1}^\infty \frac{a(n)}{n^s}$ is absolutely convergent for $\sigma>1$;
    \item\label{hp:EP} $\log L(s) = \displaystyle\sum_p\sum_{k=1}^\infty \frac{b(p^k)}{p^{ks}}$ is absolutely convergent for $\sigma>1$ and there exist $K>0$ and $\theta<1/2$ such that $|b(p^k)|\leq K p^{k\theta}$ for every prime $p$ and every integer $k\geq 1$; 
    \item\label{hp:PNT} $\sum_{p} |a(p)|p^{-\sigma}\rightarrow\infty$ when $\sigma\rightarrow 1^+$.
\end{enumerate}
Then $L(s)+L(2s)+\cdots+L(Ns)$ has infinitely many zeros for $\sigma>1$ if $N$ is sufficiently large or if $N=2$.\\
If furthermore $|a(n)|\leq n^{1/4}$ for every $n\geq 2$, then $L(s)+L(2s)+\cdots+L(Ns)$ has infinitely many zeros for $\sigma>1$ for every $N\geq 2$.
\end{theorem}

As we already mentioned the hypotheses on $L(s)$ are very mild and are satisfied by most of the known $L$-functions. In particular they are satisfied by Hecke and Artin $L$-functions by definition and Chebotarev's density theorem (see e.g. Neukirch \cite{neukirch}), by $L$-series attached to unitary cuspidal automorphic representations of $GL_r(\mathbb{A}_\QQ)$ (see Rudnick and Sarnak \cite{rudnick}), and conjecturally by elements of the Selberg class (see Selberg \cite{selberg}, and Kaczorowski and Perelli \cite{kaczper8} for remarks on \ref{hp:PNT}). In particular for Dirichlet $L$-functions $L(s,\chi)$, and the Riemann zeta function $\zeta(s)$, we know that $|a(n)|\leq 1$ for every $n\geq 1$, so we immediately get

\begin{corollary*}
For every $N\geq 2$ and every Dirichlet character $\chi$ mod ${q}$, $q\geq 1$, the Dirichlet series $L(s,\chi)+L(2s,\chi)+\cdots+L(Ns,\chi)$ has infinitely many zeros for $\sigma>1$.
\end{corollary*}

Theorem \ref{theorem:L1n} is actually an application of the following general result.

\begin{theorem}\label{theorem:sol_equation}
Let $f(s)=\sum_n c(n)n^{-s}$ be an absolutely convergent Dirichlet series for $\sigma\geq 1$ such that its logarithm is uniformly bounded for $\sigma\geq 1$, and let $L(s)$ be as in Theorem \ref{theorem:L1n}. Then $L(s)+f(s)$ has infinitely many zeros in the half-plane $\sigma>1$.
\end{theorem}

To prove Theorem \ref{theorem:sol_equation} we start from the trivial fact that a zero for $L(s)+f(s)$ is also a solution of the equation
\[
L(s)=-f(s).
\]
The hypothesis of Theorem \ref{theorem:sol_equation} are needed to allow us to take the principal branch of the logarithm of both sides, and thus we are left with the problem of finding solutions $s=\sigma+it\in\CC$ with $\sigma>1$ of the equation
\begin{equation}\label{equation:equation}
    \log L(s)=\log f(s)+\pi i,
\end{equation}
and for this we shall apply Saias and Weingartner's argument \cite{saias}.

Note that for a $L$-function $L(s)$ satisfying the hypotheses of Theorem \ref{theorem:L1n} we cannot always take the logarithm of $f(s)=L(2s)+\cdots+L(Ns)$ if $N$ is fixed since it might vanish, as the following example for $N=3$ and $L(s)=\zeta^k(s)$ shows; this is probably just a limit of the method.

\begin{theorem}\label{theorem:zetak23}
For every integer $k\geq 9$, $\zeta^k(2s)+\zeta^k(3s)$ has infinitely many zeros for $\sigma>1$. On the other hand if $1\leq k\leq 5$, $\zeta^k(2s)+\zeta^k(3s)$ has no zeros for $\sigma>1$.
\end{theorem}

The proof of this result is quite different and it is of its own interest. Essentially it follows ideas of Bohr \cite{bohr8,bohr9}, see also Titchmarsh \cite[Chapter XI]{titchmarsh}.   

\begin{acknowledgements}
The first author was partially supported by the Grant no. 2016/23/D/ST1/01149 from the National
Science Centre. The second author has been partially supported by a CRM-ISM postdoctoral fellowship and by a fellowship ``Ing. Giorgio Schirillo'' from INdAM.
\end{acknowledgements}


\section{Proof of Theorem \ref{theorem:L1n}}

\begin{lemma*}
If $L(s)$ is a Dirichlet series satisfying \ref{hp:abs_conv} and \ref{hp:EP}, then there exists $N_0\geq 3$ such that for every $N\geq N_0$
\[
|\log(L(2s)+\ldots+L(Ns))|\leq B
\]
uniformly for $\sigma\geq 1$ for some $B=B(N)>0$.\\
If furthermore $|a(n)|\leq n^{1/4}$ for every $n\geq 2$, then we can take $N_0=3$.
\end{lemma*}
\begin{proof}
Note that by \ref{hp:abs_conv} and \ref{hp:EP} it follows that for every $\eps>0$ and every $n\geq 2$ we have 
\begin{equation*}
    |a(n)|\leq C_\eps n^{\theta+\eps}
\end{equation*}
for some $C_\eps>0$. 
Hence for $\sigma\geq 1$ and $k\geq 2$, taking $\eta=1/2-\theta$ we have that
\begin{equation}\label{eq:bound}
\begin{split}
|L(ks)-1|&\leq C_{\eta}\sum_{n\geq 2} n^{-k\sigma+1/2}\leq C_{\eta}\int_{1}^\infty u^{-k\sigma+1/2}du \leq\frac{C_{\eta}}{k-3/2}.
\end{split}
\end{equation}
Therefore we have
\begin{align*}
\limsup_{n\rightarrow\infty}\bigg|\frac{L(2s)+\ldots+L(Ns)}{(N-1)}-1\bigg|&\leq \lim_{N\rightarrow\infty}\frac{1}{N-1}\sum_{k=2}^N\frac{C_{\eta}}{k-3/2}=0.
\end{align*}
Therefore the values of $L(2s)+\ldots+L(Ns)$ are located in the disc of center $s=N-1$ and radius less than $N-1$ if $N$ is sufficiently large, say $N\geq N_0$, so we can find a suitable constant $B>0$.

If we have the additional hypothesis that $|a(n)|\leq n^{1/4}$ for all $n\geq 2$, then \eqref{eq:bound} becomes
$$|L(ks)-1|\leq \frac{1}{k-5/4},$$
and thus for $N\geq 3$
$$|L(2s)+\ldots+L(Ns)-(N-1)|\leq \sum_{k=2}^N \frac{1}{k-5/4}< N-1.$$
The thesis follows with the same arguments as before.
\end{proof}

We can now prove Theorem \ref{theorem:L1n}.
\begin{proof}[Proof of Theorem \ref{theorem:L1n}]
If $N=2$, then we define $f(s)=\log(-L(2s))$ which is bounded by
$$\pi + K\sum_p\sum_{k\geq 1}p^{k(\theta-2)}\leq \pi+K \sum_p \frac{1}{p^{3/2}-1}=B<\infty$$
uniformly for $\sigma\geq 1$. Hence we can apply Theorem \ref{theorem:sol_equation}, so
$$\log(L(s))=\log(-L(2s))$$
has infinitely many solutions for $\sigma>1$, and thus $L(s)+L(2s)$ has infinitely many zeros for $\sigma>1$.

If $N\geq N_0$, where $N_0$ is given by the above lemma, then we may take $f(s)=\log(-L(2s)-\cdots-L(Ns))$, which is bounded by $B+\pi$ uniformly for $\sigma\geq 1$ by the same lemma. Hence as above we can apply Theorem \ref{theorem:sol_equation} and we deduce that $L(s)+\cdots+L(Ns)$ has infinitely many zeros for every $N\geq N_0$.
\end{proof}


\section{Proof of Theorem \ref{theorem:sol_equation}}
Let $B$ be the constant such that the logarithm of the Dirichlet series $f(s)=\sum_n c(n)n^{-s}$ is uniformly bounded by $B$. Then, by absolute convergence and Kronecker's theorem, for any fixed sequence $\{t_p\}_p\subset \RR$ and for any $\eps>0$ there exists $t\in\RR$ such that (see e.g. Perelli and Righetti \cite[Theorem 1(iv)]{perrig})
$$\left|\sum_n \frac{c(n)}{n^{\sigma}\prod_{p^\nu\| n} p^{i\nu t_p}}-f(\sigma+it)\right|<\eps$$
for every $\sigma\geq 1$. Therefore
$$\left|\log\left(\sum_n \frac{c(n)}{n^{\sigma}\prod_{p^\nu\| n} p^{i\nu t_p}}\right)\right|<B+O_B(\eps)$$
for every $\sigma\geq 1$. Taking $\eps\rightarrow0^+$ we obtain that
\begin{equation}\label{eq:unif_bound}
    \left|\log\left(\sum_n \frac{c(n)}{n^{\sigma}\prod_{p^\nu\| n} p^{i\nu t_p}}\right)\right|\leq B
\end{equation}
for every sequence $\{t_p\}_p\subset \RR$ and every $\sigma\geq 1$.

Now, as in Saias and Weingartner \cite{saias}, we use Brouwer fixed point theorem in the following way: if for some $\sigma>1$ we can solve the equation
\begin{equation} \label{equation:continuous_system}
\sum_p a(p)p^{-\sigma-it_p(z)} = z    
\end{equation}
for every $|z|\leq (B+K_\theta+\pi)$, where
$$K_\theta = K\sum_{p}\frac{1}{p^{2(1-\theta)}-p^{1-\theta}}$$
with $\theta$ and $K$ given by \ref{hp:EP}, and $t_p(z)$ are continuous real functions in the variable $z$, then we get a solution for \eqref{equation:equation}. Indeed the function
\[
z\mapsto \log\left(\sum_n \frac{c(n)}{n^{\sigma}\prod_{p^\nu\| n} p^{i\nu t_p(z)}}\right)-\sum_p\sum_{k= 2}^\infty \frac{b(p^k)}{p^{k(\sigma+it_p(z))}}+\pi i
\]
is continuous from the disk $|z|\leq (B+K_\theta+\pi)$ into itself by \eqref{eq:unif_bound} and \ref{hp:EP} for any $\sigma\geq 1$, so it has a fixed point $w$. Now, Bohr's equivalence theorem (see e.g. \cite{righetti3}) guarantees that from the above $\sigma$ and the $t_p(w)$ we can get some $s$ with $\re(s)>1$ and such that \eqref{equation:equation} holds. 

On the other hand, similarly as in Saias and Weingartner \cite{saias}, for any $\sigma>1$ such that
\begin{equation}\label{equation:suff_cond}
\sum_p |a(p)|p^{-\sigma} \geq 10(B+K_\theta+\pi)    
\end{equation} 
we have that \eqref{equation:continuous_system} holds. Indeed, let $p_1$ and $p_2$ be such that 
\[
\frac13 - \frac{1}{10\sqrt{3}}\leq \mu_1 = \frac{\sum_{p\leq p_{1}}|a(p)|p^{-\sigma}}{\sum_{p}|a(p)|p^{-\sigma}}<\frac{1}{3} \quad\hbox{ and }\quad \frac13 - \frac{1}{10\sqrt{5}}\leq\mu_2 = \frac{\sum_{p_1<p\leq p_{2}}|a(p)|p^{-\sigma}}{\sum_{p}|a(p)|p^{-\sigma}}<\frac{1}{3}.
\]
Then we have
$$\frac13 < \mu_0=1-\mu_1-\mu_2 \leq \frac13 + \frac{1}{10\sqrt{3}}+\frac{1}{10\sqrt{5}}.$$
Therefore the map $G:(0,\pi/2)^2\rightarrow \CC$, $(\theta_1,\theta_2)\mapsto\mu_{1}e^{i\theta_1}+\mu_{2}e^{-i\theta_2}$ is a diffeomorphism onto its image and, since
$$\mu_1+\mu_2-\mu_0>\frac13 - \frac{1}{5\sqrt{3}}-\frac{1}{5\sqrt{5}}\geq \frac{1}{10}\qquad \hbox{and}\qquad \mu_0-|\mu_2-\mu_1|>\frac13-\frac{1}{10\sqrt{3}}\geq\frac{1}{10},$$
we have that (cf. Figure 1 of Saias and Weingartner)
$$\hbox{Im}(G)\supset \{w\in\CC\, :\, |w-\mu_{0}|\leq 1/10 \}.$$
By \eqref{equation:suff_cond} we may take
$$w=\mu_{0}+\frac{z}{\sum_{p}|a(p)|p^{-\sigma}}$$
for any $|z|\leq (B+K_\theta+\pi)$, i.e. we may find a continuous solution
$$t_p(z)=\left\{\begin{array}{ll}
-\theta_1(z)/\log p - \arg(a(p)) & p\leq p_1\\
\theta_2(z)/\log p - \arg(a(p)) & p_1<p\leq p_2\\
 \pi/\log p - \arg(a(p)) & p> p_2
\end{array}\right.$$
of \eqref{equation:continuous_system} for any $\sigma>1$ such that \eqref{equation:suff_cond} holds.\\
Hence the result follows since \eqref{equation:suff_cond} is satisfied by any $\sigma$ sufficiently close to $1$ by \ref{hp:PNT}.


\section{Proof of Theorem \ref{theorem:zetak23}}

First of all, note that $\zeta^k(2s)+\zeta^k(3s)=0$ if and only if $$\log\frac{\zeta(3s)}{\zeta(2s)}=\frac{\pi}{k}i.$$
Moreover 
$$\log\frac{\zeta(3s)}{\zeta(2s)}=\sum_{p}\log\left(1-p^{-2s}\right)-\sum_{p}\log\left(1-p^{-3s}\right)=\sum_p\log\left(\frac{p^{3s}-p^s}{p^{3s}-1}\right).$$
Let $g(z)=\log((z^3-z)/(z^3-1))$. Then for every $r>1$ the image of the circle $|z|=r$ by means of the holomorphic function $g(z)$ is a double loop around the origin, symmetric with respect to the real axis, but thus it is not a Jordan curve (see for example Figure \ref{figure:curve}).
\begin{figure}[hbt]
    \centering
    \includegraphics[width=13cm]{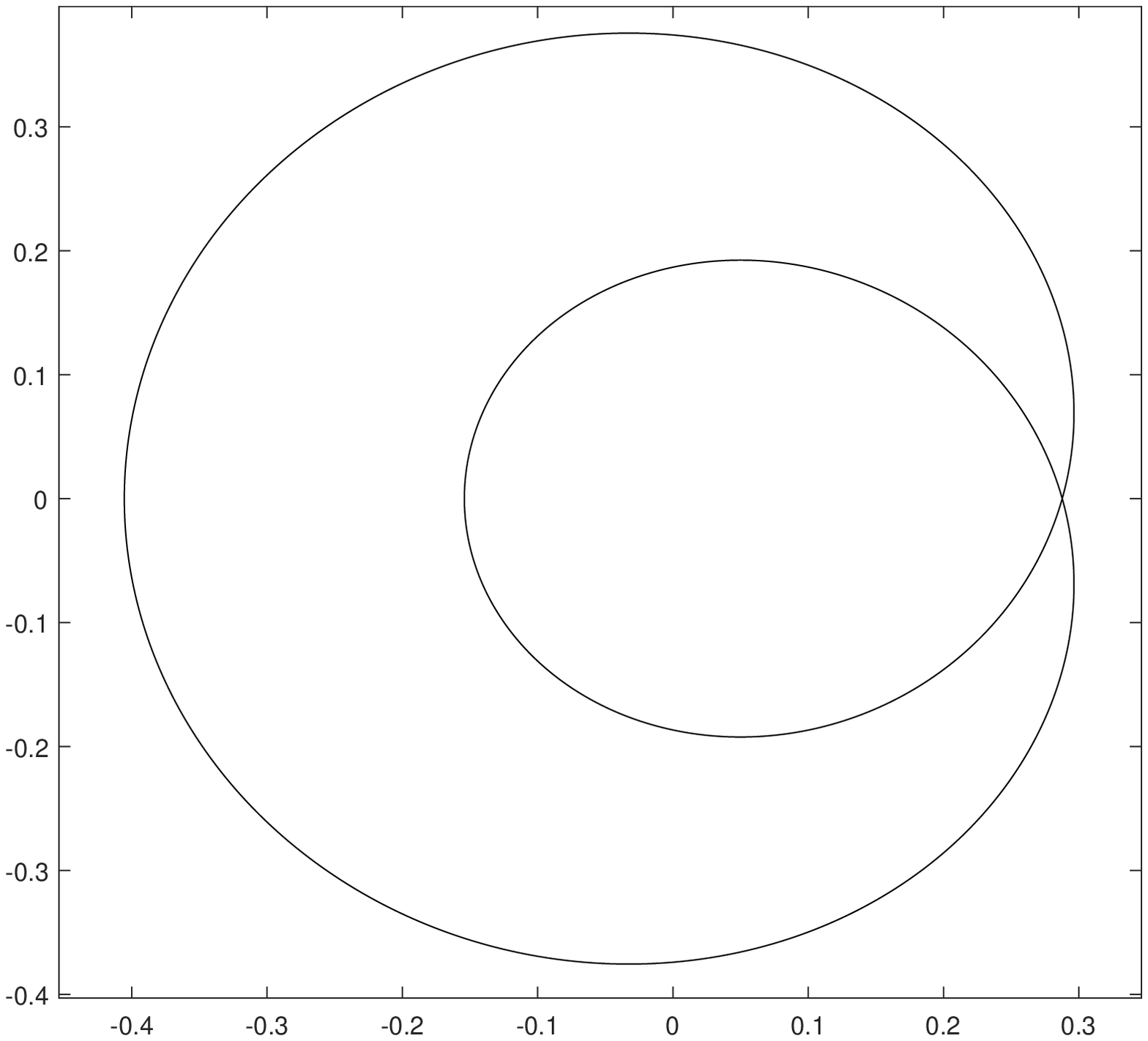}
    \caption{for $r=2$: \texttt{parametric plot (Re[log[(8 exp[3 I t]-2 exp[I t])/(8 exp[3 I t]-1)]],Im[log[(8 exp[3 I t]-2 exp[I t])/(8 exp[3 I t]-1)]])}}
    \label{figure:curve}
\end{figure}

If we write $z=re^{i\theta}$, then we have that $g(z)$ intersects the real axis four times in three distinct points:
\begin{equation}\label{equation:real_axis}
\begin{split}
&-\log\Big(\frac{r^3-1}{r^3-r}\Big)\hbox{ for }\theta=0, \quad -\log\Big(\frac{r^3+1}{r^3-r}\Big)\hbox{ for }\theta=\pi,\\
\text{ and } &\log\Big(\frac{r^2}{r^2-1}\Big)\hbox{ for }\theta=\pm 2\arctan\left(\sqrt{\frac{2r+1}{2r-1}}\right)=\pm\arccos(-1/(2r)).
\end{split}
\end{equation}
While it intersects the imaginary axis in four distinct points when
\begin{equation}\label{equation:imag_axis}
\theta=\pm 2 \arctan\!\left(\sqrt{\frac{6r^2-1}{2r^2-2r+1}\pm2r\sqrt{\frac{8r^2-3}{(2r^2-2r+1)^2}}}\right)=\pm \arccos\!\left(\frac{-1\pm \sqrt{8r^2-3}}{4r}\right)
\end{equation}
Therefore for every $r\geq 2$ we have that the image of the arc $z=re^{i\theta}$, with $|\theta|\leq \arccos(-1/(2r))$, is actually a convex Jordan curve. Indeed if we write $g(z)=u(\theta)+iv(\theta) $, then
\begin{equation*}
\begin{split}
u'(\theta)+iv'(\theta)&=i g'(z)z=i\frac{z^3-1}{z^3-z}\frac{(3z^2-1)(z^3-1)-3z^2(z^3-z)}{(z^3-1)^2}z\\
&=i\frac{\splitfrac{2r^4(\cos(2\theta)-i\sin(2\theta))+r^3(4\cos\theta-4i\sin\theta+\cos(3\theta)-i\sin(3\theta))}{+2r^2(2+\cos(2\theta)-i\sin(2\theta))+4r\cos\theta+1}}{|z^2+z+1|^2|z+1|^2}
\end{split}    
\end{equation*}
So we have
$$\frac{v'(\theta)}{u'(\theta)}=-\frac{2r^4\cos(2\theta)+r^3(\cos(3\theta)+4\cos\theta)+2r^2(2+\cos(2\theta))+4r\cos\theta+1}{2r^4\sin(2\theta)+r^3(\sin(3\theta)+4\sin\theta)+2r^2\sin(2\theta)},$$
which is strictly decreasing in $(-\arccos(-1/(2r)),\arccos(-1/(2r)))$, since its derivative is
$$-\frac{2r(8 + 31r^2 + 17r^4)\cos\theta + 4(1 + 7r^2 + 8r^4)\cos(2\theta) + r((7 + 16r^2)\cos(3\theta) + r(24 + 35r^2 + 8r^4 + 4\cos(4\theta)))}{r^2\sin^2\theta(4(r^2 + 1)\cos\theta + r(2\cos(2\theta) + 5))^2},$$
and the denominator is non-negative and the numerator is
\begin{align*}
\leq \begin{cases}
-(8r^6-14r^4-16r^3-39r^2-7r-12)<-382&\text{if } \cos\theta\geq 0,\ r\geq 2\\
-(8r^6-14r^2-16r^3-33r^2-7r-12)<-2&\text{if } -\frac{1}{2r}\leq\cos\theta<0 ,\ r\geq 2,
\end{cases}
\end{align*}
where we used the fact that in the interval we have $\cos\theta\geq -\frac{1}{2r}$ and for the latter inequality that $\cos(4\theta)>1/2$ when $-\frac{1}{2r}\leq\cos\theta<0$, $r\geq 2$.

Therefore we can apply the work of Bohr \cite{bohr1}, Haviland \cite{haviland} and Kershner \cite{kershner1} on the addition of convex Jordan curves to a subset of the set $A=\{\log(\zeta(2\sigma+2it)/\zeta(3\sigma+3it)):\sigma> 1,t\in\RR\}.$ In fact if we denote for every prime $p$ and any $\sigma\geq 1$ with $\mathcal{C}_{p,\sigma}$ the image via $g(z)$ of the arc $z=p^{\sigma}e^{i\theta}$, with $|\theta|\leq \arccos(-1/(2p^\sigma))$, then by Bohr's equivalence theorem the vectorial addition $\mathcal{R}_\sigma=\sum_p \mathcal{C}_{p,\sigma}$ is contained in $A$ for every $\sigma>1$. By Bohr \cite[\S3]{bohr1} we know that $\mathcal{R}_\sigma$ is either the region of plane inside a convex Jordan curve or the region in between two convex Jordan curves. We can actually say more: let $h_{p,\sigma}(\theta)$ denote the supporting function of the curve $\mathcal{C}_{p,\sigma}$, i.e.
$$h_{p,\sigma}(\theta)=\max_{z\in\mathcal{C}_{p,\sigma}}\re(e^{-i\theta}z).$$
Then by Haviland \cite[Theorem IIB]{haviland} the supporting function of the outer curve is $h_{O,\sigma}(\theta)=\sum_p h_{p,\sigma}(\theta)$, while by Kershner  \cite[Theorems II and III]{kershner1} the supporting function of the inner curve (if any) satisfies $h_{I,\sigma}(\theta)\leq h_{2,\sigma}(\theta)-\sum_{p\geq 3} h_{p,\sigma}(\theta+\pi)$. Hence by \eqref{equation:real_axis} we get in particular that $A$ contains any point $x$ on the real axis such that
$$-\sum_{p\geq 2} \log\Big(\frac{p^{3\sigma}-1}{p^{3\sigma}-p^\sigma}\Big)=h_{O,\sigma}(\pi)e^{i\pi}\leq x\leq  -\log\Big(\frac{2^{3\sigma}-1}{2^{3\sigma}-2^{\sigma}}\Big)+\sum_{p\geq 3} \log\Big(\frac{p^{2\sigma}}{p^{2\sigma}-1}\Big)\leq h_{I,\sigma}(\pi)e^{i\pi},$$
for some $\sigma>1$. Since for $\sigma$ sufficiently close to $1$ the LHS is negative and the RHS is positive, then $A$ contains a neighborhood of the origin. Hence, since the outer and the possibly inner curve are convex, then $A$ must contain any point $yi$ on the imaginary axis such that
$$|y|\leq h_{O,\sigma}(\pi/2) = \sum_{p\geq 2} g(p^\sigma e^{i\theta_{p,\sigma}}),\quad\hbox{where } \theta_{p,\sigma}=\arccos\Big(\frac{-1+\sqrt{8p^{2\sigma}-3}}{4p^\sigma}\Big)$$
by \eqref{equation:imag_axis} for some $\sigma>1$. Since for $\sigma$ sufficiently close to 1 the RHS is greater than $0.36$, we conclude that $\frac{\pi}{k}i\in A$ if $k\geq 9$.

On the other hand, by considering for example the convex hulls of the image via $g(z)$ of the circle $|z|= p$, and taking into account the above considerations on the outer curve of the vectorial addition, it is clear that there cannot be points $yi$ in $A$ such that $y$ is real and
$$|y|> \sum_{p\geq 2} g(pe^{i\theta_p'})> 0.61966,\quad\hbox{where } \theta_p'=-\arccos\Big(\frac{-1-\sqrt{8p^{2\sigma}-3}}{4p^\sigma}\Big).$$
Hence $\zeta^k(2s)+\zeta^k(3s)$ has no zeros for $\sigma>1$ if $k\leq 5$.

\bibliographystyle{amsplain}
\bibliography{biblio}

\end{document}